\documentclass[12pt,reqno]{amsart}
\usepackage{amsmath,amssymb,amsthm,mathtools,wasysym,calc,verbatim,enumitem,tikz,pgfplots,url,mathrsfs,fullpage,bbm}
\usepackage{bbold}
\usepackage{comment}

\newtheorem{mytheo}{Theorem}[section]
\newtheorem{mydef}[mytheo]{Definition}
\newtheorem{myrem}[mytheo]{Remark}
\newtheorem{mylem}[mytheo]{Lemma}
\newtheorem{mycoro}[mytheo]{Corollary}
\newtheorem{mypropo}[mytheo]{Proposition}

\theoremstyle{definition}

\theoremstyle{remark}

\usepackage[colorlinks = true,
            linkcolor  = blue,
            citecolor  = blue,
            urlcolor   = blue,
            bookmarks  = true]{hyperref}

\theoremstyle{remark}

\numberwithin{equation}{section}

\usepackage[backend=biber, style=numeric, citestyle=numeric-comp]{biblatex}
\DeclareFieldFormat{labelnumberwidth}{\mkbibbrackets{#1}}
\DeclareNameAlias{author}{family-given}

\DeclareFieldFormat*{title}{#1}

\AtEveryBibitem{%
  \clearfield{issn}
  \clearfield{url}
  \clearfield{eprint}
}

\renewbibmacro{in:}{}

\DeclareFieldFormat[article]{volume}{\textbf{#1}}
\DeclareFieldFormat[article]{number}{\textbf{#1}}
\renewbibmacro*{volume+number+eid}{%
  \printfield{volume}
  \iffieldundef{number}
    {}%
    {\printtext[parens]{\printfield{number}}}
  \setunit{\addcomma\space}%
  \printfield{eid}}

\DeclareFieldFormat[journal]{journaltitle}{\mkbibemph{#1}}

\DeclareNameAlias{author}{family-given}
\DeclareNameAlias{editor}{family-given}

\DeclareDelimFormat{namedelim}{\addcomma\space}
\DeclareFieldFormat[article,incollection,inproceedings,book]{giveninit}{\mkbibnamegiven{#1}}

\DeclareNameFormat{family-given}{%
  \usebibmacro{name:family-given}
    {\namepartfamily}
    {\namepartgiveni} 
    {\namepartprefix}
    {\namepartsuffix}%
  \usebibmacro{name:andothers}}

\DeclareFieldFormat{pages}{#1}

 \allowdisplaybreaks[4]

\begin{document}

\title{Rank deficiency of Bernoulli random matrices for growing corank}

\author{Zeyan Song}
\address{Shandong University, Jinan, 250100, China.}
\email{zeyansong8@gmail.com}

\author{Hanchao Wang }
\address{Shandong University, Jinan, 250100, China.}
\email{wanghanchao@sdu.edu.cn}
\subjclass[2020]{15B52, 60B20}

\begin{abstract}
        Let $A$ be an $n \times n$ Bernoulli random matrix whose entries are i.i.d. Bernoulli($p$) random variables with $p \le 1/2$. In this paper, we determine the probability that the corank of $A$ is at least $k$ when $k$ is of order $o(\sqrt{\log n})$
        \begin{align*}
            \mathbb{P}\left( \text{corank}A \ge k \right) = \left(1-p+o_{n}(1) \right)^{kn}.
        \end{align*} 

\end{abstract}

\maketitle

\section{Introduction}\label{Introduction}
 Let $A$ be a uniformly distributed random matrix on $\{0,1\}^{n \times n}$. The study of the probability that $A$ has a large corank is a central topic in random matrix theory. In particular, the probability that the corank is at least one, namely, the probability that the matrix is singular, has been extensively studied. In the 1960s, Koml\'{o}s proved that $\mathbb{P}(\det(A)=0) = o_n(1)$. Consider that the probability of $A$ has a zero column, the natural conjecture is as follows. 
 \begin{align}\label{Singularity}
     \mathbb{P}\left( \det{A} = 0 \right) = \left( \frac{1}{2}+o_{n}(1) \right)^{n}.
 \end{align}

Much later, the first exponential upper bound $\mathbb{P}(\det(A)=0) \le 0.998^{n} $ was obtained by Kahn, Koml\'{o}s and Szemer\'{e}di in \cite{KKS_jams}. Through the following observation, 
\begin{align*}
    \mathbb{P}\left( \det{A}=0 \right)=\mathbb{P}\left( \exists x \in \mathbb{R}^{n}\setminus\{ 0\} :\Vert Ax \Vert_{2}=0   \right),
\end{align*}
where $\Vert \cdot\Vert_{2}$ is the Euclidean norm in $\mathbb{R}^{n}$, we can reduce the problem of invertibility of a random matrix to estimating the probability that the inner product between a random vector and a given deterministic vector equals zero. In particular, letting $a:=(a_{1},\dots,a_{n}) \in \mathbb{R}^{n}$ and $X:=(X_{1},\dots,X_{n})$ be a uniformly distributed random vector on $\{0,1\}^{n}$, we consider the upper bound of
\begin{align}\label{Littlewood-Offord problem}
    \mathbb{P}\left(  a_{1}X_{1}+\dots+a_{n}X_{n}=0 \right).
\end{align}
Problems of this type are now commonly known as Littlewood-Offord problems. Through a detailed investigation of these problems, Tao and Vu \cite{TV_rsa, TV_jams} improved the exponential upper bound to $(3/4+o_{n}(1))^{n}$. In 2010, Bourgain, Vu, and Wood further improved the upper bound to $\left(1/\sqrt{2} + o_n(1)\right)^n$ in \cite{BVW_jfa}.

Meanwhile, the study of the smallest singular value $\sigma_{n}(A)$ is of fundamental importance. Since a matrix is singular precisely when its smallest singular value vanishes, the invertibility problem is closely connected to estimating $\sigma_{n}(A)$. In fact, the geometric method introduced by Rudelson and Vershynin \cite{RV_adv} shows that these two questions are, to a large extent, equivalent. They established that
\begin{align*}
    \mathbb{P}\left( \sigma_{n}(A) \le \varepsilon n^{-1/2} \right) \le C\varepsilon +e^{-cn},
\end{align*}
where $C$ and $c$ are absolute constants. In addition, a significant contribution of their work was to describe the relationship between the Littlewood-Offord problem and the arithmetic structure of the vector $a$ in terms of its least common denominator (LCD). More precisely, they demonstrated that there are only a few vectors $a$ for which the upper bound in \eqref{Littlewood-Offord problem} is large. This viewpoint may be regarded as an alternative formulation of the Littlewood-Offord problem: given a prescribed level of small ball probability, one seeks to understand the structural properties that $a$ must satisfy or to estimate the number of such vectors. 

This viewpoint is commonly referred to as the inverse Littlewood–Offord problem. For further developments of this theory and its applications to random matrix theory, see, for example \cite{TV_aom,Nguyen_duke}. 

Following the strategy of first fixing the magnitude of the probability and then estimating the number of vectors, Tikhomirov \cite{Tikhomirov} ultimately completed the proof of \eqref{Singularity}. More specifically, he introduced a new method, referred to as the ``inversion of randomness", through which he was able to show that the number of lattice vectors satisfying a certain arithmetic structural property (namely, those for which the probability exceeds a given threshold) is super-exponentially small. 

We now return to the problem of estimating the probability that $A$ has
large corank.  Recall that the event
\begin{align*}
    \mathrm{rank}A\le n-k
\end{align*}
is equivalent to the assertion that the columns of $A$ are contained in
some subspace of dimension at most $n-k$. Thus, this problem may be viewed
as a large-deviation version of the singularity problem.  In their seminal
work on the singularity probability, Kahn, Komlós and Szemerédi
\cite{KKS_jams} already obtained a superexponential estimate of the
form
\begin{align*}
    \mathbb{P}\left( \mathrm{rank}A \le n-k \right) \le o_k(1)^n .
\end{align*}

On the other hand, the elementary event that $k+1$ prescribed rows of $A$
are identical has probability $2^{-kn}$ and already forces
$\mathrm{rank}A\le n-k$.  This suggests that the correct exponential
rate should be $2^{-kn}$.  In this direction, Vu \cite{Vu_crmt} formulated
the conjecture
\begin{align}\label{rank}
    \mathbb{P}\left( \mathrm{rank}A \le n-k  \right)
    = \left( \frac{1}{2}+o(1) \right)^{kn}.
\end{align}

When $k$ is fixed, this conjectural exponential rate was obtained by
Jain, Sah and Sawhney \cite{Jain_ecp}, combining their argument with
Tikhomirov's method for the singularity problem.  Their result identifies the
sharp rate
\begin{align*}
    \mathbb{P}\left( \mathrm{corank} A \ge k \right)
    =2^{-(k+o(1))n},
\end{align*}
for every fixed $k$.  However, the methods in this fixed-corank regime are
not uniform enough to yield the conjectural estimate when $k$ is allowed to
grow with $n$.  In the growing-corank regime, the best general estimates
have therefore been of the weaker form
\begin{align}\label{rank-with-c}
    \mathbb{P}\left( \mathrm{rank}A \le n-k  \right)
    \le e^{-ckn},
\end{align}
for some absolute constant $c>0$.

The first result of this type was due to Rudelson \cite{Rudelson_aop}, who
proved \eqref{rank-with-c} for random matrices with independent identically
distributed non-constant subgaussian entries in the range
$k\le c\sqrt n$.  Rudelson's proof belongs to the geometric approach to
invertibility and combines distance estimate ideas.

More recently, Hunter, Kwan, Sauermann and Sawhney
\cite{HKSS_arxiv} removed the restriction $k\le \sqrt n$ for random sign matrix and proved \eqref{rank-with-c} for the full range $1\le k\le n$.  A key feature of their proof is that it returns to the earlier \cite{KKS_jams} strategy rather than the geometric framework.  Their main new input is a high-dimensional relative anticoncentration inequality. This line of work has also led to extensions beyond the Bernoulli model.
Polavarapu \cite{Polavarapu_arxiv} adapted the framework in \cite{HKSS_arxiv} to independent, not necessarily identically distributed, real entries that satisfy a uniform atom bound
\begin{align*}
    \sup_{x\in\mathbb R}\mathbb P(A_{ij}=x)\le b<1.
\end{align*}
Using a Bernoulli decomposition of the entries and reorganizing the thin/thick
argument in a genuinely inhomogeneous setting, this gives
\begin{align*}
    \mathbb{P}\left( \mathrm{rank}A\le n-k \right)\le \exp(-c_b kn),
    \qquad 1\le k\le n,
\end{align*}
with a constant $c_b>0$ depending only on $b$. 

Thus, exponential upper bounds of the form $\exp(-\Omega(kn))$ are now available under broad assumptions and, in the Bernoulli setting, throughout the full range $1 \le k \le n$. The objective of the present paper is different. Rather than seeking an estimate that is uniform over the widest possible range of $k$, our objective is to determine the sharp exponential asymptotics of the large corank probability in a regime where $k$ is allowed to grow with $n$. It also main result in this paper.

In this paper, we investigate the precise probability that the matrix has rank deficiency $k$ in the regime where $k$ grows with $n$. More specifically, we prove that \eqref{rank} holds for $k=o(\sqrt{\log{n}})$.

\begin{mytheo}\label{Theorem A}
   For every $p \in (0,1/2]$ and $\varepsilon > 0$, there are $n_{\ref{Theorem A}}$ and $c_{\ref{Theorem A}} >0$ depending only on $p$ and $\varepsilon$ with the following property. Let $n \ge n_{\ref{Theorem A}}$, $1 \le k \le c_{\ref{Theorem A}}\sqrt{\log{n}}$, and let $B_{n}(p)$ be $n \times n$ random matrix with independent entries $b_{ij}$ such that $\mathbb{P}(b_{ij}=0)=1-p$ and $\mathbb{P}(b_{ij}=1)=p$. Then we have 
   \begin{align*}
       \mathbb{P}\left(  \mathrm{rank}\left(B_{n}(p) \right) \le n-k   \right) \le \left( 1-p+\varepsilon \right)^{kn}.
   \end{align*}
\end{mytheo}
\begin{myrem}
Considering the event that $B_{n}(p)$ has $k$ zero columns, we obtain
    \begin{align*}
        \mathbb{P}\left(  \mathrm{corank}(B_{n}(p)) \ge k \right) = \left( 1-p + o_{n}(1)\right)^{kn}.
    \end{align*}
    In particular, when $p = 1/2$, conjecture \eqref{rank} holds for $k =o(\sqrt{\log{n}})$.
\end{myrem}
\subsection{Proof Strategy and Main Innovations}\label{Proof sketch}
We briefly describe the main ideas and innovations of the proof. We divide the argument into three main steps.

First, we make a simple observation. Note that if the corank of $ A$ is at least $k$, then there exist $k$ column vectors that are linearly dependent on the remaining $n-k$ column vectors. For example, if we take the first $k$ columns, then $ A_{1}, \ldots, A_{k}$ belong to the linear span of $A_{k+1}, \ldots, A_{n}$. Let $B$ be the $(n-k)\times n$ matrix whose rows are $A_{k+1}^{\top},\dots,A_n^{\top}$. Then we obtain the following estimate
\begin{align*}
    \mathbb{P}\left( \text{corank}A \ge k \right)\le \binom{n}{k}  \mathbb{P}\left( \forall v \in \ker{B}: \left\langle  A_{i},v \right\rangle=0 \text{ for all } i \in [k]\right).
\end{align*}

Therefore, it suffices to analyze the properties of the vector $v$. Based on the classical decomposition of the unit sphere introduced by Rudelson and Vershynin \cite{RV_adv}, we split the argument into two parts: the compressible (Comp) vectors and the incompressible (Incomp) vectors.

Since $\dim(\ker B) \ge k$, our second step is to show that the probability that $k$ orthogonal vectors are all compressible is $(1-p+o(1))^{kn}$. In particular, suppose that $A$ has $k$ identically zero rows. Then $\ker B$ contains $k$ standard basis vectors, all of which belong to $\mathrm{Comp}$. Consequently, the contribution of the leading-order to probability essentially comes from the Compressible vectors. Consequently, we deduce that there must exist at least one incompressible vector in $\ker B$, that is, among any $k$ orthogonal vectors, at least one is incompressible. Then, by invoking the inversion of randomness technique developed by Tikhomirov \cite{Tikhomirov}, we are able to complete the proof. In fact, for both the Comp and Incomp parts, we introduce several new methods that refine the existing approach.

For the compressible part, we obtain essentially the best possible result. In fact, we derive a sharp upper bound for the probability of $k=o(n)$. Consequently, this method is applicable to establish \eqref{rank} for a much wider range of $k$.

To achieve a precise probability estimate, the standard Hanson-Wright type inequalities are no longer sufficient. Instead, we construct a new random vector to replace the Bernoulli vector and establish a small ball probability bound for this newly constructed vector. This, in turn, allows us to deduce the desired small ball probability estimate for the original Bernoulli random vector. Indeed, this constitutes the crucial innovation that allows us to push the range of $k$ beyond fixed constants and permit it to increase with $n$. The remaining upper bound on $k$ arises entirely from the analysis of the incompressible component.

For the incompressible part, we rely on Tikhomirov’s inversion of randomness method in \cite{Tikhomirov}. We refine his argument to its optimal form in our setting. In fact, if one only works with a single incompressible vector and does not consider higher-dimensional vector systems, the restriction $k \le \sqrt{\log n}$ turns out to be necessary.

If instead one attempts to work with higher-dimensional systems of vectors, one is naturally led to a high-dimensional Littlewood-Offord problem. However, at present there is no available method that yields a sufficiently strong estimate for this setting; in particular, there is no known way to prove that the number of $k$ orthogonal incompressible unit vectors with comparatively large small ball probability is super-exponentially small of order $e^{-Mkn}$.

Indeed, if one characterizes the high-dimensional Littlewood-Offord problem via the least common denominator (LCD) without imposing additional structural assumptions, the best possible bound on the number of such vectors is only exponential of order $ e^{-cn}$. At the same time, to obtain the precise probability, the approaches developed in \cite{Rudelson_aop} or \cite{HKSS_arxiv} break down immediately.

\textbf{Organization of this paper}
The paper is organized as follows. In Section \ref{Notation}, we introduce the notation and collect several basic definitions that will be used throughout the paper. Section \ref{Preliminaries} contains preliminary results. In Section \ref{Compressible vectors}, we analyze the compressible component. There we establish essentially optimal probability bounds valid for arbitrary $k$, which already yield sharp estimates in this regime. Section \ref{Proof of main theorem} is devoted to the proof of the main theorem. In particular, we treat the incompressible component through a refinement of Tikhomirov’s inversion of randomness method, which leads to the restriction $k \le c\sqrt{\log{n}}$. Combining compressible and incompressible analyzes, we complete the proof of Theorem \ref{Theorem A}. The detailed proof for the incompressible part will be deferred to the appendix.

\section{Notation}\label{Notation}
We denote by $[n]$ the set of natural numbers from $1$ to $n$. Given a vector $x\in \mathbb{R}^{n}$, we denote by $\Vert x\Vert_{2}$ its standard Euclidean norm: $\Vert x\Vert_{2}=\left(\sum_{j\in [n]}{x_{j}^{2}} \right)^{\frac{1}{2}}$, and the supnorm is denoted $\Vert x\Vert_{\infty}=\max_{i}{|x_{i}|}$. The unit sphere of $\mathbb{R}^{n}$ is denoted by $S^{n-1}$. The cardinality of a finite set $\mathrm{I}$ is denoted by $\left| \mathrm{I} \right|$.
\par 
If $V$ is a $m\times l$ matrix, we denote by $\mathrm{Row}_{i}\left(V\right)$ its $i$-th row and $\mathrm{Col}_{j}\left(V\right)$ its $j$-th column. Its singular values will be denoted by 
$$s_{1}\left(V\right) \ge s_{2}\left(V\right) \ge \cdots \ge s_{m}\left(V\right) \ge 0.$$
The Euclidean operator norm of $V$ is defined as 
$$\Vert V\Vert =\max_{x\in S^{n-1}}\Vert Vx\Vert_{2},$$
and the Hilbert-Schmidt norm as 
\begin{align*}
    \Vert V\Vert_{\mathrm{HS}}=\left(\sum_{i=1}^{m}\sum_{j=1}^{l}v_{i,j}^{2}\right)^{\frac{1}{2}}.
\end{align*}
Note that $\Vert V\Vert=s_{1}\left(V\right)$ and $\Vert V\Vert_{\mathrm{HS}}=\left(\sum_{j=1}^{m}s_{j}\left(V\right)^{2}\right)^{\frac{1}{2}}$.

We denote by $\mathcal{L}(X,t)$ the L\'{e}vy concentration function of a random vector $X \in \mathbb{R}^{m}$:
\begin{align*}
    \mathcal{L}(X,t)=\sup_{y \in \mathbb{R}^{m}}{\mathbb{P}\left( \Vert X-y \Vert_{2} \le t \right)}.
\end{align*}

Denote by $c$, $c',\dots$ the universal constants and by $c\left( u\right)$, $C\left( u\right)$ the constants depending only on $u$. Their value can change from line to line.

\section{Preliminaries}\label{Preliminaries}

We will need to estimate the number of integer points in a ball in $\mathbb{R}^{n}$. The set $B\left(0,R\right)$ is the ball of radius R centered at 0.

\begin{mylem} \label{Integer size in ball}
For any $R>0$,
\begin{align*}
    \left| \mathbb{Z}^{n} \cap RB_{2}^{n}\right| \le \left(2+\frac{C_{\ref{Integer size in ball}}R}{\sqrt{n}}\right)^{n},
\end{align*}
where $C_{\ref{Integer size in ball}}>0$ is an absolute constant.
\end{mylem}

We will introduce auxiliary result concerning random variables.

\begin{mylem}\label{Tenzorization}Let $X_{1},\dots,X_{n}$ be independent non-negative random variables
\begin{itemize}
    \item Assume that there exist $\eta >0$ and $ \tau >0$ such that $\mathbb{P}( X_{j} \le \eta) \le \tau$. Then for all $\varepsilon \in (0,1]$,
    \begin{align*}
        \mathbb{P}\left( \sum_{j=1}^{n} X_{j} \le \eta \varepsilon n \right) \le \left( \frac{e}{\varepsilon} \right)^{\varepsilon n} \tau^{n-\varepsilon n}.
    \end{align*}
\item  Assume that there exist $M$, $m> 0$ and $s_{0} \ge 0$ such that $\mathbb{P}\left(X_{j}\le s\right)\le \left(Ms \right)^{m}$ for all $s\ge s_{0}$. Then
\begin{align*}
\mathbb{P}\left(\sum_{j=1}^{n}X_{j} \le n s\right)\le \left(C_{\ref{Tenzorization}}M s\right)^{mn}\quad ~\text{for all}~~ s\ge s_{0},
\end{align*}
where $C_{\ref{Tenzorization}}$ is a constant.
\end{itemize}

\end{mylem}

Next, we introduce an important definition in the field of non-asymptotic random matrix theory, which was originally proposed by Rudelson and Vershynin \cite{RV_adv}.

\begin{mydef}\label{Comp and Incomp}
Let $\delta,\rho \in \left( 0,1 \right)$ and $n \in \mathbb{Z}^{+}$, we define the sets of sparse, compressible and incompressible vectors as follows:

\begin{itemize}
   \item  $\mathrm{Sparse}_{n}\left( \delta \right)=\left\{ x \in \mathbb{R}^{n} : \left| \mathrm{supp}\left( x \right)\right| \le \delta n \right\};$
    \item $\mathrm{Comp}_{n}\left( \delta,\rho \right)=\left\{x \in S^{n-1} : \mathrm{dist}\left( x,\mathrm{Sparse}_{n}\left( \delta \right) \right) \le \rho \right\};$
    \item $\mathrm{Incomp}_{n}\left( \delta,\rho \right)=S^{n-1}\setminus \mathrm{Comp}_{n}\left( \delta,\rho \right)$.
\end{itemize}
 
\end{mydef}

\section{Compressible vectors}\label{Compressible vectors}

The goal of this section is to prove that for an $(n-k)\times n$ Bernoulli$(p)$ matrix $B$, with overwhelming probability, any collection of $k$ orthonormal vectors in $\ker B$ cannot all be compressible.

The following theorem is the main result of this section.
\begin{mypropo}\label{Compressible subspace}
    For any $p \in (0,1/2]$ and $\varepsilon >0$, there exist $n_{\ref{Compressible subspace}}$ and $\tau_{\ref{Compressible subspace}}$ depending only on $p$ and $\varepsilon$ such that for all $n \ge n_{\ref{Compressible subspace}}$ and $\log{n} \ge k \ge 1$. Define an event $\mathcal{E}_{\mathrm{comp}}$ as the event that there exist orthonormal vectors $x_{1},\dots,x_{k} \subset \ker{B}$ satisfying $x_{1},\dots,x_{k} \in  \mathrm{Comp}_{n}(\tau_{\ref{Compressible subspace}}^{2},\tau_{\ref{Compressible subspace}}^{4})$. Then
    \begin{align*}
        \mathbb{P}\left( \mathcal{E}_{\mathrm{comp}} \right) \le \left( 1-p+\varepsilon \right)^{kn}.
    \end{align*}
\end{mypropo}

The first step of the proof is to construct an appropriate $\varepsilon$-net for the set of $k$ orthonormal vectors that are compressible. In fact, this construction follows the approach introduced by Rudelson \cite{Rudelson_aop} via Random rounding. We begin with the following definition and then describe the construction of the net.

\begin{mydef}\label{Almost orthogonal}
Let $\nu \in \left( 0,1 \right)$. An l-tuple of vectors $\left( v_{1},v_{2},\dots,v_{l} \right) \subset \mathbb{R}^{n} \setminus \left\{ 0 \right\}$ is called $\nu$-almost orthogonal if the $n \times l$ matrix $V_{0}$ with $\mathrm{Col}_{j}\left( V_{0} \right)=\frac{v_{j}}{\Vert v_{j} \Vert_{2}} $ satisfying
$$1-\nu \le s_{l}\left( V_{0} \right) \le s_{1}\left( V_{0} \right) \le 1+\nu.$$
\end{mydef}

\begin{mylem}[Proposition 4.2 in \cite{Rudelson_aop}]\label{Sparse net}
    Let $v_1,\dots,v_{l} \in \mathrm{Comp}_{n}(\tau^{2},\tau^{4})$ be an orthogonal system. Then there exists $u_1,\dots,u_{l} \in \mathrm{Sparse}_{n}(4\tau^{2})\cap\frac{\tau}{\sqrt{n}}\mathbb{Z}^{n}\cap \frac{3}{2}B_{2}^{n}\setminus \frac{1}{2}B_{2}^{n} $ are $\frac{1}{2}$-almost orthogonal satisfying 
    \begin{align*}
        \Vert B(v_{j}-u_{j})\Vert_{2} \le 3\tau\sqrt{n} \text{ for all } j \in [l].
    \end{align*}
\end{mylem}

In addition, we need to estimate the small ball probability on the constructed net.
More precisely, we need to bound the probability that the Euclidean norm of $Wb$ is small, where $W$ is an almost orthogonal $k \times n$ matrix and $b$ is an $n$-dimensional Bernoulli$(p)$ random vector. This amounts to a high-dimensional Littlewood-Offord problem.

Standard tools such as Hanson-Wright type inequalities do not provide sufficiently sharp probability bounds for our purposes. Instead, we employ a refined small ball estimate. This approach begins with the following small ball probability lemma, due to Rudelson and Vershynin \cite{RV_imrn}.

\begin{mylem}[Corollary 1.4 in \cite{RV_imrn}]\label{Small ball probability for linear image}
    Consider a random vector $X = (X_1,...,X_n)$ where $X_i$ are real-valued independent random variables. Let $t, b \ge  0$ be such that
\begin{align*}
    \mathcal{L}\left( X_{i}, t \right) \le b \text{ for all } i \in [n]
\end{align*}
 Let $P$ be an orthogonal projection in $\mathbb{R}^{n}$ onto a $d$-dimensional subspace. Then
 \begin{align*}
     \mathcal{L}\left( PX, t\sqrt{d}  \right) \le \left( C_{\ref{Small ball probability for linear image}}b \right)^{d},
 \end{align*}
 where $C_{\ref{Small ball probability for linear image}}>0$ is a absolute constant.
\end{mylem}

In fact, if the constant $C_{\ref{Small ball probability for linear image}}$ were arbitrarily close to $1$, this would already yield the desired result. Our strategy is to introduce a new auxiliary random variable to replace the Bernoulli$(p)$ distribution, which allows us to improve the constant $C_{\ref{Small ball probability for linear image}}$ so that it becomes close to $1$. This constitutes the main idea of the proof in this section.

We now present an almost optimal small ball probability bound for Bernoulli random variables.

\begin{mylem}\label{Small ball probability for Comp vectors}
    Let $m, n \in \mathbb{N}$ be such that $m \le n$ and let $M$ be an $m \times n$ matrix with independent Bernoulli($p$) entries. Then for all $\varepsilon >0$, there exist $c_{\ref{Small ball probability for Comp vectors}} >0$ depend on $p$ and $\varepsilon$ such that for all $\frac{1}{2}$-almost orthogonal system $v_1, \dots,v_l \in S^{n-1}$, we have
    \begin{align*}
        \mathbb{P}\left( \Vert Mv_{j}\Vert_{2} \le c_{\ref{Small ball probability for Comp vectors}} \sqrt{m} \text{ for all } j \in [l]   \right) \le (1-p+\varepsilon)^{lm}.
    \end{align*}
\end{mylem}
\begin{myrem}\label{SBP remark}
    The exponent in Lemma \ref{Small ball probability for Comp vectors} is optimal uniformly over all almost orthogonal systems. Indeed, for $v_{j}=e_{j}, j \in [l]$, the event that the first $l$ columns of $M$ vanish has probability $(1-p)^{lm}$ and implies $Mv_{j}=0$ for every $j \in [l]$.
\end{myrem}
\begin{proof}
    Let $V=(v_1,\dots,v_{l})$ be an $n \times l$ matrix with $1/2 \le s_{l}(V) \le  s_{1}(V) \le 3/2 $, and $B=(b_{1},\dots,b_{n})$ be random vectors with independent Bernoulli($p$) entries. We first estimate the upper bound of the following probability:
    \begin{align*}
        \mathbb{P}\left( \Vert V^{\top}B\Vert_{2} \le \delta\sqrt{l}  \right)
    \end{align*}
    Note that there exist $u_1,\dots,u_l \in S^{n-1} $ orthogonal vectors such that for all vectors $X \in \mathbb{R}^{n}$
    \begin{align*}
        \Vert V^{\top}X\Vert_{2}^{2} \ge \frac{1}{4}\sum_{i=1}^{l}\left\langle u_{i},X \right\rangle^{2} \ge \Vert PX\Vert_{2}^{2}/4,
    \end{align*}
    where $P:=\sum_{i=1}^{l}u_i u_i^{\top}$ is an orthogonal projection in $\mathbb{R}^{n}$ onto a $l$-dimensional subspace.

    Let $B_{i}=(b^{(i)}_{j})_{j \in [n]}, i \in [\mu]$ be random vectors with independent Bernoulli($p$) entries, and let $B_{i}$ and $B_{j}$ be independent. Set $a_i = 2^{i-1}$ for every $i \in [\mu]$. We have 
    \begin{align*}
        \mathbb{P}\left( \Vert V^{\top}B\Vert_{2} \le \delta\sqrt{l}  \right)
        & \le \mathbb{P}\left( \Vert V^{\top}a_{i}B_{i}\Vert_{2} \le \delta a_{i}\sqrt{l} \text{ for all } i \in [\mu] \right)^{1/\mu} \\
        & \le \mathbb{P}\left( \Vert PX\Vert_{2} \le 2^{\mu+1} \delta\sqrt{\mu l}  \right)^{1/\mu},
    \end{align*}
    where $X:= \left( x_1,\dots,x_{n}  \right)=\left(\sum_{j=1}^{\mu}a_{j}b_{1}^{(j)},\dots, \sum_{j=1}^{\mu}a_{j}b_{n}^{(j)}   \right)$.

    It is not difficult to note that for all $i \le n$
    \begin{align*}
        \mathcal{L}\left( x_{i},1/2 \right) \le (1-p)^{\mu}.
    \end{align*}
    Let $\mu :=  \frac{\log{C_{\ref{Small ball probability for linear image}}}}{\log{ (1+\varepsilon_{0})}}  $ and $\delta = 2^{-\mu-2}\mu^{-1/2}$, applying Lemma \ref{Small ball probability for linear image}, we have  
    \begin{align*}
        \mathbb{P}\left(  \Vert V^{\top}B\Vert_{2} \le \delta\sqrt{l} \right) \le \left( C_{\ref{Small ball probability for linear image}}(1-p)^{\mu}  \right)^{l/\mu} \le \left( 1-p+\varepsilon_{0} \right)^{l}
    \end{align*}
    Applying the above inequality for every row of $M$, we obtain
    \begin{align*}
        \mathbb{P}\left( \Vert V^{\top} \mathrm{row}_{i}(M)\Vert_{2} \le \delta \sqrt{l} \right) \le (1-p+\varepsilon_{0})^{l}.
    \end{align*}
    Finally, we use the Lemma \ref{Tenzorization} to have 
    \begin{align*}
        \mathbb{P}\left( \Vert Mv_{j}\Vert_{2} \le \nu \delta \sqrt{m} \text{ for } j \in [l]   \right) 
        & \le  \mathbb{P}\left( \sum_{i=1}^{m}\Vert V^{\top}\mathrm{row}_{i}(M)\Vert_{2}^{2} \le \nu^{2} \delta^{2} ml \right) \\
        & \le \left( \frac{e}{\nu^{2}} \right)^{\nu^{2} m} (1-p+\varepsilon_{0})^{(1-\nu^{2})lm}.
        \end{align*}
        It remains to note that by choosing $\nu$ to be sufficiently small and $\varepsilon_{0}=\varepsilon/2$, we can complete the proof of this result.
\end{proof}

\begin{proof}[\textsf{Proof of the Proposition \ref{Compressible subspace}}]
    Combining Lemma \ref{Integer size in ball}, Lemma \ref{Sparse net} and Lemma \ref{Small ball probability for Comp vectors}, we obtain that for all $\varepsilon >0$, there exists $c$ depending only on $\varepsilon$ such that for all $\tau \le c$:
    \begin{align*}
        \mathbb{P}\left( \mathcal{E}_{\text{comp}} \right) \le \left(2+\frac{C_{\ref{Integer size in ball}}}{\tau}   \right)^{4\tau^{2}nk}\left( 1-p+\varepsilon \right)^{k(n-k)}. 
    \end{align*}
   The result follows by choosing $\tau$ sufficiently small.

\end{proof}
\section{Proof of main theorem}\label{Proof of main theorem}
Before completing the proof of the main result, we briefly recall the inversion of randomness technique introduced by Tikhomirov \cite{Tikhomirov}.
Roughly speaking, this method allows one to convert the structural information about vectors into probabilistic estimates.
In particular, Tikhomirov \cite{Tikhomirov} considered the following threshold function.
\begin{mydef}\label{threshold function}
    For $p \in (0,1/2]$, $L \ge 1$, and $x \in \mathbb{R}^{n}$, let $b_{1},\dots,b_{n}$ be independent Bernoulli($p$) random variables and we define 
    \begin{align*}
        \mathcal{T}_{p}(x,L):=\sup\left\{ t \in (0,1): \mathcal{L}(\sum_{i=1}^{n}b_{i}x_{i},t) > Lt \right\}.
    \end{align*}
\end{mydef}

Observe that this function essentially captures the Littlewood-Offord problem.
We establish the following property, which shows that the threshold function associated with the incompressible part is sufficiently small.
Since the proof only involves a minor modification of the argument of Tikhomirov \cite{Tikhomirov}, we include here a brief proof of the following proposition for the reader's convenience, while postponing the proof of the key lemma to the appendix.

\begin{mypropo}\label{Incomp Levy}
    Let $\delta,\rho,\varepsilon \in (0,1)$ and $k \ge 1$, there exist $n_{\ref{Incomp Levy}}$, $L_{\ref{Incomp Levy}}$ and $c_{\ref{Incomp Levy}}$ depending on $\delta,\rho,\varepsilon,p$ such that for all $n\ge n_{\ref{Incomp Levy}}$ and $1 \le k \le c_{\ref{Incomp Levy}}\sqrt{\log{n}}$, with the probability at least $1-(1-p)^{4kn}$, for all $x \in \mathrm{Incomp}_{n}(\delta,\rho) \cap \ker{B}$:
    \begin{align*}
        \mathcal{T}_{p}(x,L_{\ref{Incomp Levy}}) \le \left( 1-p+\varepsilon \right)^{n}.
    \end{align*}
\end{mypropo}
\begin{proof}
Let $\varepsilon_0\in(0,\varepsilon)$ be sufficiently small such that $(1-p+\varepsilon_{0})^{n}\sqrt{n} \le (1-p+\varepsilon)^{n}$ and set $q: = 1-p + \varepsilon_0$.

It is enough to prove that, with probability at least $1-(1-p)^{4kn}$, there is no vector
\begin{align*}
    x\in \mathrm{Incomp}_n(\delta,\rho)\cap \ker B
\end{align*}
such that
\begin{align*}
\mathcal{T}_p(x,L)>q^n,
\end{align*}
where $L$ is a sufficiently large constant dependent only on
$\delta,\rho,\varepsilon,p$. Indeed, since $q<1-p+\varepsilon$, this immediately implies
\begin{align*}
\mathcal{T}_p(x,L)\le (1-p+\varepsilon)^n
\end{align*}
for all sufficiently large $n$.

We decompose the possible values of $\mathcal{T}_{p}(x,L)$. Let
\begin{align*}
\mathcal R:=\{2^j q^n: j\ge 0,\ 2^j q^n<1\}.
\end{align*}
Then $|\mathcal R|\le Cn$, where $C=C(p,\varepsilon_0)$. Hence it suffices to prove the following estimate: for every $r\in \mathcal R$,
\begin{align}\label{One scale}
\mathbb P\left(
\exists x\in \mathrm{Incomp}_n(\delta,\rho)\cap \ker B:
r<\mathcal{T}_p(x,L)\le 2r
\right)
\le \exp(-Mkn),
\end{align}
where $M>0$ will be chosen sufficiently large depending only on $p,\delta,\rho$.

We now prove \eqref{One scale}. Fix $r\in\mathcal R$. Let
\begin{align*}
N=N(r):=\left\lceil r^{-1}\right\rceil .
\end{align*}
By the standard spread lemma for incompressible vectors in \cite{RV_adv}, there exist constants
$\delta_0=\delta_0(\delta,\rho)>0$ and $c_0,C_0>0$, depending only on
$\delta,\rho$, such that for every
$x\in\mathrm{Incomp}_n(\delta,\rho)$ there is a subset
$I=I(x)\subset[n]$ with $|I|\ge \delta_0 n$ and
\begin{align*}
\frac{c_0}{\sqrt n}\le |x_i|\le \frac{C_0}{\sqrt n},
\qquad i\in I.
\end{align*}
After a suitable permutation of the coordinates, we shall assume throughout the sequel that
\begin{align*}
I\supset [\lfloor \delta_0 n\rfloor].
\end{align*}

We use the usual random rounding to obtain that for every
$ \{ x \in\mathrm{Incomp}_n(\delta,\rho)\cap \ker B: \mathcal{T}_{p}(x,L) \in [r,2r]\} $, there exists an $(N,n,K,\delta_0)$-admissible set
\begin{align*}
\mathcal A_{r}=A_1\times\cdots\times A_n\subset \mathbb Z^n,
\end{align*}
where $K=K(\delta,\rho)$ such that the following hold.

For all $x \in\{ x \in\mathrm{Incomp}_n(\delta,\rho)\cap \ker B: \mathcal{T}_{p}(x,L) \in [r,2r]\}$, there exists $y \in \mathcal{A}_{r}$ such that
\begin{align*}
 C_{1}LN^{-1} \ge  \mathcal{L}\left(\sum_{i=1}^n b_i y_i,\sqrt n\right)
\ge c_1 L N^{-1},
\end{align*}
where $c_1=c_1(\delta,\rho)>0$. Moreover, $\|B(y-N\sqrt{n}x)\|_{2} \le n$. 

Let $\mathcal{N}_{r} \subset \mathcal{A}_{r}$ be the above net, applying Corollary \ref{Corollary of inversion}, for all $1 \le k \le c_{\ref{Corollary of inversion}}(p,M)\sqrt{\log{n}}$, we have 
\begin{align*}
    |\mathcal{N}_{r}| \le e^{-Mkn}|\mathcal{A}_{r}| \le e^{-Mkn}(KN)^{n}.
\end{align*}
Thus, we have 
\begin{align*}
\mathbb P\left(
\exists x\in \mathrm{Incomp}_n(\delta,\rho)\cap\ker B:
r<\mathcal{T}_p(x,L)\le 2r
\right)
\le \mathbb P\left( \exists y \in \mathcal{N}_{r}: \|By\|_{2} \le n  \right) \le e^{-Q(M,k)n}
\end{align*}
where $Q(M,k)=Mk-C(\delta,\rho)$.
Let $M$ be large enough, we have 
\begin{align*}
    \mathbb P\left(
\exists x\in \mathrm{Incomp}_n(\delta,\rho)\cap\ker B:
r<\mathcal{T}_p(x,L)\le 2r
\right)
\le (1-p)^{5kn}.
\end{align*}

Finally, for all $1 \le k \le c\sqrt{\log{n}}$ and $n \ge n_{0}$, summing over all $r\in\mathcal R$, we get
\begin{align*}
\mathbb P\left(
\exists x\in \mathrm{Incomp}_n(\delta,\rho)\cap\ker B:
\mathcal{T}_p(x,L)>q^n
\right)
\le Cn(1-p)^{5kn}.
\end{align*}
For all sufficiently large $n$, the right-hand side is bounded by
\begin{align*}
(1-p)^{4kn}.
\end{align*}
Therefore, with probability at least $1-(1-p)^{4kn}$, every
\begin{align*}
x\in\mathrm{Incomp}_n(\delta,\rho)\cap\ker B
\end{align*}
satisfies
\begin{align*}
\mathcal{T}_p(x,L)\le q^n\le (1-p+\varepsilon)^n.
\end{align*}
Taking $L_{\ref{Incomp Levy}}=L$ and choosing
\begin{align*}
c_{\ref{Incomp Levy}}\le c_{\ref{Corollary of inversion}}(p,M,\delta,\rho,\varepsilon)
\end{align*}
small enough, the restriction $k\le c_{\ref{Incomp Levy}}\sqrt{\log n}$ is exactly the range in which
Corollary \ref{Corollary of inversion} applies. This completes the proof.
\end{proof}

We have now completed all the necessary preparations and are ready to prove our main result.

\begin{proof}[\textsf{Proof of the Theorem \ref{Theorem A}}]
    
    Recalling the discussion in Subsection \ref{Proof sketch}, let $B$ be an $(n-k) \times n$ Bernoulli($p$) random matrix. We observe that if $\mathrm{corank}(A) \ge k$, then there exist $k$ column vectors that lie in the linear span of the remaining $n-k$ column vectors. Taking into account the number of ways to choose $k$ columns out of $n$, we obtain the following estimate
    \begin{align*}
        \mathbb{P}\left( \text{corank}A \ge k  \right) \le \binom{n}{k}\mathbb{P}\left( \forall v \in \ker{B}:\left\langle A_{i},v \right\rangle=0 \text{ for all } i \in [k]  \right).
    \end{align*}
    
    Since $\dim(\ker B) \ge k$, we can find a collection of $k$ orthonormal unit vectors in $\ker B$. According to Proposition \ref{Compressible subspace}, if $B$ lies in the event $\mathcal{E}_{\text{comp}}^{c}$, then among these $k$ orthonormal unit vectors there must be at least one incompressible unit vector. Consequently, we further obtain
    \begin{align*}
        & \mathbb{P}\left( \forall v \in \ker{B}:\left\langle A_{i},v \right\rangle=0 \text{ for all } i \in [k]  \right) \\
        & \le\mathbb{P}\left(  v(B) \in \ker{B} \cap \mathrm{Incomp}\left( \tau_{\ref{Compressible subspace}}^{2},\tau_{\ref{Compressible subspace}}^{4}\right):\left\langle A_{i},v(B) \right\rangle=0 \text{ for all } i \in [k] \right) +\left( 1-p+\varepsilon \right)^{kn},
    \end{align*}
    where $\tau_{\ref{Compressible subspace}}:=\tau_{\ref{Compressible subspace}}(p,\varepsilon)$ and $v(B) \in \mathbb{R}^{n}$ is a random vectors depending only on $B$. 

    Finally, combining Proposition \ref{Incomp Levy} with the independence of the vectors $A_{1},\dots,A_{k}$, we have
    \begin{align*}
        \mathbb{P}\left( \text{corank}A \ge k \right) \le \binom{n}{k}\left( L_{\ref{Incomp Levy}}(1-p+\varepsilon)^{n} \right)^{k}+ 2\binom{n}{k} \left(1-p+\varepsilon  \right)^{kn} \le \left( 1-p+2\varepsilon \right)^{kn},
    \end{align*}
    where $L_{\ref{Incomp Levy}}:=L_{\ref{Incomp Levy}}(\tau_{\ref{Compressible subspace}}^{2},\tau_{\ref{Compressible subspace}}^{4},p,\varepsilon)$ and $n$ is larger than a constant depending only on $p$ and $\varepsilon$. We have now completed the final proof.

\end{proof}


\textbf{Acknowledgment:} This work was supported by the National Key R\&D Program of China (No.2024YFA1013501), the National Natural Science Foundation of China  (No. 12571162),   Shandong Provincial Natural Science Foundation (No. ZR2024MA082), and the Youth Student Fundamental study Funds of Shandong University  (No. SDU-QM-B202407).

\printbibliography

\appendix
\section{Inversion of randomness}
In the appendix, we briefly explain how ``Inversion of Randomness" argument in Tikhomirov \cite{Tikhomirov} can be used to improve the upper bound of the probability to match that in Proposition \ref{Incomp Levy}. 

We first introduce the basic framework of Tikhomirov’s approach. Let $N, n \ge 1$ be some integers, and let $\delta \in (0,1]$ and $K \ge 1$ be some real numbers. 
We say that a subset $\mathcal{A} \subset \mathbb{Z}^n$ is $(N,n,K,\delta)$-admissible if

\begin{itemize}
    \item $\mathcal{A} = A_1 \times A_2 \times \cdots \times A_n$, where every $A_i$ $(i=1,2,\ldots,n)$ is an origin-symmetric subset of $\mathbb{Z}$;
    
    \item $A_i$ is an integer interval of cardinality at least $2N+1$ for every $i > \delta n$;
    
    \item $A_i$ is a union of two integer intervals of total cardinality at least $2N$ and 
    $A_i \cap [-N,N] = \varnothing$ for all $i \le \delta n$;
    
    \item $|A_1| \cdot |A_2| \cdots |A_n| \le (KN)^n$;
    
    \item $\max A_i < nN$ for all $1 \le i \le n$.
\end{itemize}
Let $\mathcal{A} = A_1 \times A_2 \times \cdots \times A_n \subset \mathbb{Z}^n$ 
be an $(N,n,K,\delta)$-admissible set, and let $f(t)$ be any real-valued function on $\mathbb{Z}$. 
Fix any $p \in (0,1)$, and assume that $X_1, X_2, \ldots, X_n$ are independent 
integer random variables, where each $X_i$ is uniform in $A_i$. 
For every $\ell \le n$, we define a random function $f_{\mathcal{A},p,\ell}$ by

\begin{align*}
f_{\mathcal{A},p,\ell}(t) 
:= \mathbb{E}_b\, f\!\left(t + \sum_{j=1}^{\ell} b_j X_j \right)
= \sum_{(v_j)_{j=1}^{\ell} \in \{0,1\}^{\ell}}
p^{\sum v_j} (1-p)^{\ell - \sum v_j}
f\!\left(t + v_1 X_1 + \cdots + v_\ell X_\ell \right),
\end{align*}
$t \in \mathbb{Z}$, where $\mathbb{E}_b$ denotes the expectation with respect to 
the randomness of the vector $b = (b_1, \ldots, b_n)$ with independent 
Bernoulli$(p)$ components. Now, we give the main result in this appendix.

\begin{mypropo}\label{Main result in appendix}
For any $\delta \in (0,1]$, $p \in (0,1/2]$, $\varepsilon \in (0,p)$, 
$K,M \ge 1$ there are 
$n_{\ref{Main result in appendix}} = n_{\ref{Main result in appendix}}(\delta,\varepsilon,p,K,M) \ge 1$ depending on 
$\delta,\varepsilon,p,K,M$, $L_{\ref{Main result in appendix}} = L_{\ref{Main result in appendix}}(\delta,\varepsilon,p,K) > 0$ depending only on 
$\delta,\varepsilon,p,K$ (and not on $M$) and $c_{\ref{Main result in appendix}}:=c_{\ref{Main result in appendix}}(p,M,\delta,\varepsilon)$ depending only on $p,M,\delta,\varepsilon$ with the following property.

Take $n \ge n_{\ref{Main result in appendix}}$, $1\le k \le c_{\ref{Main result in appendix}}\sqrt{\log{n}}$, $1 \le N \le (1-p+\varepsilon)^{-n}$, 
let $\mathcal{A}$ be an $(N,n,K,\delta)$-admissible set and 
$f(t)$ be a non-negative function in $\ell_1(\mathbb{Z})$ 
with $\|f\|_1 = 1$ and such that $\log_2 f$ is $n^{-1/2}$-Lipschitz. 
Then, with $f_{\mathcal{A},p,n}$ defined above, we have
\begin{align*}
    \mathbb{P}\left( 
\|f_{\mathcal{A},p,n}\|_\infty 
> L_{\ref{Main result in appendix}} (N\sqrt{n})^{-1}
\right)
\le \exp(-Mkn).
\end{align*}
\end{mypropo}

According to the proof of Corollary 4.3 in Tikhomirov \cite{Tikhomirov}, the indicator function can be approximated by a function $f$ that satisfies the assumptions of the above proposition. Consequently, $f_{\mathcal{A},p,n}$ can be viewed as a L\'{e}vy concentration function. In particular, we have the following corollary
\begin{mycoro}\label{Corollary of inversion}
    For any $\delta \in (0,1]$, $p \in (0,1/2]$, $\varepsilon \in (0,p)$, 
$K,M \ge 1$ there are 
$n_{\ref{Corollary of inversion}} = n_{\ref{Corollary of inversion}}(\delta,\varepsilon,p,K,M) \ge 1$ depending on 
$\delta,\varepsilon,p,K,M$, $L_{\ref{Corollary of inversion}} = L_{\ref{Corollary of inversion}}(\delta,\varepsilon,p,K) > 0$ depending only on 
$\delta,\varepsilon,p,K$ (and not on $M$) and $c_{\ref{Corollary of inversion}}:=c_{\ref{Corollary of inversion}}(p,M,\delta,\varepsilon)$ depending only on $p,M,\delta,\varepsilon$ with the following property.

Take $n \ge n_{\ref{Corollary of inversion}}$, $1\le k \le c_{\ref{Corollary of inversion}}\sqrt{\log{n}}$, $1 \le N \le (1-p+\varepsilon)^{-n}$, 
let $\mathcal{A}$ be an $(N,n,K,\delta)$-admissible set. Further, assume that $b_{1},\dots,b_{n}$ are i.i.d. Bernoulli($p$) random variables. Then
\begin{align*}
    \left| \left\{  x \in \mathcal{A}: \mathcal{L}\left( \sum_{i=1}^{n}{b_{i}x_{i}},\sqrt{n} \right)  \ge L_{\ref{Corollary of inversion}}N^{-1} \right\} \right|
\le \exp(-Mkn)|\mathcal{A}|.
\end{align*}
\end{mycoro}
\begin{proof}
    Let $n\ge  n_{\ref{Main result in appendix}}$. Following the argument of the proof of Corollary 4.3 in \cite{Tikhomirov},  define the function $f \in \ell_{1}(\mathbb{Z})$ as
    \begin{align*}
        f(t):= \frac{1}{m_{0}}2^{-|t|/\sqrt{n}}, \  t \in \mathbb{Z}
    \end{align*}
    where $m_{0}:= \sum_{t \in \mathbb{Z}}{2^{-|t|/\sqrt{n}}}$. Note that $\| f\|_{1}=1$, $\log_{2}f(t)$ is $n^{-1/2}$-Lipschitz and $f(t) \ge \frac{c}{\sqrt{n}}\textbf{1}_{[-\sqrt{n}-1,\sqrt{n}+1]}(t)$.

    Applying Proposition \ref{Main result in appendix} and the properties of $f(t)$, we have 
    \begin{align*}
        \left| \left\{  x \in \mathcal{A}: \sup_{t \in \mathbb{Z}}\mathbb{E}\textbf{1}_{[-\sqrt{n}-1,\sqrt{n}+1]}\left(t+ \sum_{i=1}^{n}{b_{i}x_{i}} \right)  \ge CL_{\ref{Main result in appendix}}N^{-1} \right\} \right|
\le \exp{(-Mkn)}|\mathcal{A}|.
    \end{align*}
    Furthermore, we have
    \begin{align*}
        \left| \left\{  x \in \mathcal{A}: \mathcal{L}\left( \sum_{i=1}^{n}{b_{i}x_{i}},\sqrt{n} \right)  \ge L_{\ref{Corollary of inversion}}N^{-1} \right\} \right|
\le \exp(-Mkn)|\mathcal{A}|.
    \end{align*}
\end{proof}
Our argument follows Section 4 in Tikhomirov \cite{Tikhomirov} almost verbatim. We make suitable modifications so that the upper bound of the probability resulting improves to $e^{-Mkn}$. We now briefly outline this procedure, beginning with a preliminary result.

\begin{mypropo}\label{One dimensional 1}
    For any $M>0$, $p \in (0,1/2]$, $\delta \in (0,1)$ and $\varepsilon \in (0,p)$. there exist $n_{\ref{One dimensional 1}}=n_{\ref{One dimensional 1}}(p,\delta,\varepsilon)$ and $L_{\ref{One dimensional 1}}= C_{\ref{One dimensional 1}}e^{c_{\ref{One dimensional 1}}Mk}$, where $C_{\ref{One dimensional 1}}$ and $c_{\ref{One dimensional 1}}$ depending only on $p,\varepsilon$ and $\delta$. Let $f \in \ell_{1}(\mathbb{Z})$ such that $\| f\|_{1}=1$, $n \ge n_{\ref{One dimensional 1}}$, $n/2 \le l \le n$ and let $\mathcal{A}$ be an $(N,n,K,\delta)$-admissible set for $N \le 2^{n}$ and $K>0$. Then
    \begin{align*}
        \mathbb{P}\left( \Vert f_{\mathcal{A},p,l}\Vert_{\infty} \ge \max \{ L_{\ref{One dimensional 1}}/N\sqrt{n} , \left( 1-p+\varepsilon \right)^{l}\Vert f\Vert_{\infty} \} \right) \le e^{-Mkn}.
    \end{align*}
\end{mypropo}
\begin{proof}
    By the proof of Proposition 4.5 in \cite{Tikhomirov}, we complete the proof of this theorem.
\end{proof}
\begin{mypropo}\label{one dimensional 2}
For any $p \in (0,1/2]$, $\gamma \in (0,1)$, $\widetilde R > 1$, 
$L_0 \ge 16\widetilde R$ and $M \ge 1$ there are $n_{\ref{one dimensional 2}} = n_{\ref{one dimensional 2}}(p,L_0,\widetilde R,M) > 0$, $\eta_{\ref{one dimensional 2}} = C_{\ref{one dimensional 2}}\exp{ \left( -\frac{Mk}{\gamma}  \right)}$, where $C_{\ref{one dimensional 2}}$ is an universal constant, $c_{\ref{one dimensional 2}}=c_{\ref{one dimensional 2}}(p,M)$ and $1 \le k \le c_{\ref{one dimensional 2}}\gamma \log{n}$ with the following property. Let $L_0 \ge L \ge 16\widetilde R$, let $n \ge n_{\ref{one dimensional 2}}$, $N \le 2^n$, let $g \in \ell_1(\mathbb{Z})$ be a non-negative function satisfying
\begin{itemize}
    \item $\|g\|_1 = 1$;
    
    \item $\log_2 g$ is $\eta_{\ref{one dimensional 2}}$-Lipschitz;
    
    \item $\displaystyle \sum_{t \in I} g(t) \le \frac{\widetilde R}{\sqrt{n}}$
    for any integer interval $I$ of cardinality $N$;
    
    \item $\displaystyle \|g\|_\infty \le \frac{L}{N\sqrt{n}}$.
\end{itemize}

For each $i \le \lfloor \gamma n \rfloor$, let $X_i$ be a random variable 
uniform on some disjoint union of integer intervals of cardinality at least $N$ each; 
and assume that $X_1,\dots,X_{\lfloor \gamma n \rfloor}$ are independent. 
Define a random function $\widetilde g \in \ell_1(\mathbb{Z})$ as
\begin{align*}
\widetilde g(t) 
& := \mathbb{E}_b\, g\!\left(t + \sum_{i=1}^{\lfloor \gamma n \rfloor} b_i X_i \right)\\
& = \sum_{(v_i)_{i=1}^{\lfloor \gamma n \rfloor} \in \{0,1\}^{\lfloor \gamma n \rfloor}}
p^{\sum_i v_i} (1-p)^{\lfloor \gamma n \rfloor - \sum_i v_i}
g\!\left(t + v_1 X_1 + \cdots + v_{\lfloor \gamma n \rfloor} X_{\lfloor \gamma n \rfloor}\right),
\end{align*}
where $b = (b_1,\dots,b_n)$ is the vector of independent Bernoulli$(p)$ components. Then
\begin{align*}
\mathbb{P}\!\left\{
\|\widetilde g\|_\infty 
> \frac{(p/\sqrt{2}+1-p)L}{N\sqrt{n}}
\right\}
\le \exp(-Mkn).
\end{align*}
\end{mypropo}
\begin{proof}
    In fact, the proof of this theorem is essentially identical to that of Proposition 4.10 in \cite{Tikhomirov}. It therefore suffices to verify that the variables involved satisfy the required assumptions. Following the proof of Proposition 4.10 in \cite{Tikhomirov}, after suitable adjustments, we see that it is enough to require that $n$ satisfies the following two conditions:
    The first is the Lipschitz condition, which requires $\eta_{\ref{one dimensional 2}} > n^{-1/2}$.
    
    The second is the requirement that $c \eta_{\ref{one dimensional 2}} \gamma n > L_{\ref{one dimensional 2}} \sqrt{n}$, where $c$ is constant depending only on $p$. Thus, we only need $k \le c\gamma \log n$.
\end{proof}

\begin{proof}[\textsf{Proof of the Proposition \ref{Main result in appendix}}]
    As the last step, we only need to note that $\gamma = \beta(\delta,p,\varepsilon,M)/k$ to decay the $L_{\ref{One dimensional 1}}$. Thus, $k^{2} \le c\log n$.
\end{proof}

\end{document}